\documentclass[10pt,leqno,a4paper]{article}

\usepackage[ansinew]{inputenc}
\usepackage{amsmath}
\usepackage{amsthm}
\usepackage{amssymb}

\usepackage{hyperref}
\usepackage[english]{babel}
\usepackage[totalheight=22 true cm, totalwidth=14 true cm]{geometry}
\usepackage{setspace}
\usepackage{mathrsfs}
\usepackage{xypic}

\newtheorem{thm}{Theorem}[section]
\newtheorem{cor}[thm]{Corollary}
\newtheorem{lemma}[thm]{Lemma}
\newtheorem{prop}[thm]{Proposition}

\newtheorem{defn}[thm]{Definition}

\theoremstyle{remark}

\theoremstyle{definition}

\newtheorem{rmk}[thm]{Remark}

\newtheorem{sit}[thm]{Situation}
\numberwithin{equation}{thm}

\def\beq{\begin{equation}}
\def\eeq{\end{equation}}

\def\crash#1{}
\def\N{{\mathbb N}}
\def\Z{{\mathbb Z}}

\def\R{{\mathbb R}}

\def\A{{\mathbb A}}

\def\l{\left}
\def\r{\right}
\def\[[{\l[\l[}
\def\]]{\r]\r]}
\def\p{\prime}

\def\cf{\emph{cf. }}
\def\ie{\emph{i.e. }}

\def\ds{\displaystyle}

\def\cE{{\mathcal E}}

\def\cM{{\mathcal M}}

\def\cO{{\mathcal O}}
\def\cH{{\mathcal H}}

\def\cL{{\mathcal L}}

\def\cX{{\mathcal X}}
\def\cY{{\mathcal Y}}

\def\sA{{\mathscr A}}
\def\sB{{\mathscr B}}

\def\sF{{\mathscr F}}
\def\sH{{\mathscr H}}
\def\sK{{\mathscr K}}

\def\sM{{\mathscr M}}
\def\sS{{\mathscr S}}

\def\wtilde{\widetilde}
\def\what{\widehat}

\def\veps{\varepsilon}
\def\a{\alpha}
\def\be{\beta}

\def\la{\lambda}

\def\na{\nabla}
\def\d{{\partial}}
\def\di{{\frac{\partial}{\partial x_i}}}

\def\op{{\rm op}}

\def\Spf{{\rm Spf}}

\def\ker{{\rm Ker}}
\def\ul{\underline}

\def\kc{{k^\circ}}

\def\kt{\tilde{k}}

\author{Francesco Baldassarri\thanks{Universit\`{a} di Padova,
Dipartimento di matematica pura e applicata, Via Trieste, 63, 35121 Padova, Italy.}
\and Lucia Di Vizio\thanks{Institut de Math\'{e}matiques de Jussieu,
Topologie et g\'{e}om\'{e}trie alg\'{e}briques, Case 7012,
2, place Jussieu, 75251 Paris Cedex 05, France.}}

\title{Continuity of the radius of convergence of $p$-adic differential equations on
Berkovich analytic spaces}

\begin{document}

\maketitle


\tableofcontents

\section{Introduction}

The
classical existence  theorem of Cauchy  \cite[Chap.I]{Poole}
for local solutions of an analytic differential system at an ordinary point
does not hold in general for differential equations on a smooth
Berkovich
analytic space $X$ over a $p$-adic
field $k$. We recall \cite[1.2.2]{Berkovich} that to any point
$\xi \in X$ one associates a completely valued extension
field ${\sH}(\xi)$ of $k$, called the {\it residue field at $\xi$}; the point $\xi$ is $k$-rational if $\sH(\xi) = k$.
Any $k$-rational point $\xi$ of $X$ admits a neighborhood isomorphic to a polydisk centered at the origin $O$ in an affine (analytic) $k$-space, the isomorphism sending $\xi$ to $O$.
However,
the neighborhoods of a non-rigid point are in general
too coarse. So,   a differential equation does in general have no solutions analytic in  a full neighborhood of  a non-rigid point $\xi \in X$, even if the point is not a singularity of the equation.
In the very inspiring  paper \cite[\S 3]{AndreOrbifold}
Y. Andr\'e concentrates on differential equations
which after pull-back to a finite \'etale covering admit a full set
of multivalued analytic solutions. For such differential equations
there is a notion of global monodromy group close to the one in the complex case.
It would be interesting to pursue Andr\'e's investigation into a description of integrable analytic connections locally for the \'etale topology of \cite{BerkovichEtale}. But this is not our approach here: we  use the natural topology on Berkovich analytic spaces and regard an \'etale covering $f:Y \to X$ as producing a {\it highly non-trivial connection} $(f_{\ast} \cO_Y, \na = f_{\ast}(d_{X/k}) : f_{\ast}\cO_Y \to f_{\ast}\cO_Y \otimes \Omega^1_X)$ on $X$. Moreover,
the problem of the failure of Cauchy existence theorem would not be overcome in general by using some \'etale topology.  On the other hand,
it is possible and sometimes convenient to recover
Cauchy's theorem at any given point $\xi \in X$, by performing the extension
of scalars to $X \what {\otimes} \sH(\xi)$, and passing to some canonical point $\xi^\p$ of this space above $\xi$. This viewpoint has been systematically used by Dwork and
Robba in their study of $p$-adic differential equations.
\par
We actually  assume that $X$ comes with a local notion of distance, measured in terms of an embedding of $X$ as an analytic domain in the generic fiber $\cX_{\eta}$ of a  smooth formal scheme $\cX$ over $\kc$.  This does not mean that we  privilege formal schemes over $\kc$ or $\kt$-schemes, over $k$-analytic spaces. The formal model of $X$ is here a technical tool for expressing ``local" radii of convergence of solutions of differential equations in the above sense. We will show in a subsequent paper that certain expressions in these local radii are in fact absolute invariants of a connection on an analytic space.
\par
In practice, we  consider all over this paper the following
\begin{sit}\label{situation}The smooth formal scheme $\cX = \Spf A$, is affine and  \'etale over $\what \A^d_\kc$, the formal affine space over $\kc$, with formal coordinates $\ul x =(x_1,\dots,x_d)$. Then $\sA :=  A\otimes k$, $X = \cM(\sA)$, and $U$ is an  analytic domain in $X$.
\end{sit}
We are given  an integrable system of partial differential equations of the form
\beq\label{eq:intrdiffsys}
\Sigma = \Sigma_{(\ul x,  \ul G, U)} \;\; : \;\; \frac{\partial \, \vec y}{\partial x_i}= G_i\,  \vec y \;\; ,\ \forall\ i=1,\dots,d\,,
\eeq
for $\vec y$ a column vector of unknown functions and  $G_i$ a $\mu \times \mu$ matrix of analytic functions
on $U$.  Notice that, for any $k$-rational point $\xi \in X$, the \'etale map $\ul x : X \to \A^d_k$ to the affine $k$-analytic space of dimension $d$, admits a unique local section $\sigma_{\xi}: D_k^d(\ul x (\xi),1^-) \to X$, sending $\ul x (\xi)$ to $\xi$, where
\beq
D_k^d(\ul x (\xi),1^-) = \{ \eta \in \A_k^d  \, | \, |x_i(\eta) - x_i(\xi)| < 1 \, , \,  {\rm for} \, 1=1,\dots,d \, \} \; .
\eeq
The image of $\sigma_{\xi}$ will be denoted $D_{\cX}(\xi,1^-)$, and called {\it the open disk of radius 1 centered at the $k$-rational point $\xi \in X$}. Similarly, we define open and closed disks $D_{\cX}(\xi,r^\pm)$, of radius $r <1$ centered at $\xi$. Notice that we use the term ``disk" to refer to ``polydisk with equal radii".
The {\it diameter} $\delta_{\cX}(\xi,U)$ of $U$ at the $k$-rational point $\xi$, is the radius of the maximal open disk centered at $\xi$ and contained in $U$, a notion obviously independent of the choice of the formal coordinates $\ul x$ on $\cX$.  Then $0 < \delta_{\cX}(\xi,U) \leq 1$ because, on the one hand, a $k$-rational point of $U$ is necessarily an interior point of $U$ in $X$; on the other hand, disks of radii $\geq 1$ are not defined.
Now, when $\xi \in U$ is a $k$-rational point of $U$, the definition of the {\it radius of convergence of the system
(\ref{eq:intrdiffsys}) at $\xi$} is completely natural.
It is the radius $r=  R_{\cX}(\xi,\Sigma)= R(\xi,\Sigma)$ of the maximal open  disk $D_{\cX}(\xi,r^-)$ contained in $U$, where
 a fundamental solution matrix $Y$ of (\ref{eq:intrdiffsys}) at $\xi$ converges.  Notice that $Y$ is a matrix
 with entries in  $k[[x_1-x_1(\xi),\dots,x_d-x_d(\xi)]]$ and its convergence is really tested in $D_k^d(\ul x(\xi),1^-)$.
If
\beq \label{eq:Radiusseries} Y =  \sum_{\underline \a \in \N^d} A_{\underline\a}
(x_1 - x_1(\xi))^{\a_1}  \cdots (x_d - x_d(\xi))^{\a_d} \; , \;  {\rm with} \,  A_{\ul \a} \in M_{\mu}(k)\; ,
 \eeq
 its radius of convergence is, as in the classical case,
\beq \label{eq:locliminf}
\wtilde R(\xi,\Sigma) = \liminf_{|\underline\a|_{\infty} \to
\infty} \l|A_{\underline\a}\r|^{- {1 / |\underline\a|_{\infty}}} \in \R_{\geq 0}
\cup \{ \infty \} \; ,
\eeq
where
$|\underline\a|_{\infty} =\a_1 + \dots+\a_d$, and where the norm of a matrix is the maximum absolute value of its entries. Notice that the disk of radius $\wtilde R(\xi,\Sigma)$, centered at $\xi \in \A_k^d$, is not necessarily contained in $U$, as the example of the trivial connection $G_i = 0$, $\forall i$, on a small disk $U \subset \A_k^d$ shows. But we insist on defining
\beq \label{eq:radiusdef}
R(\xi,\Sigma) = \min (\wtilde R(\xi,\Sigma),\delta_{\cX}(\xi,U)) \; .
\eeq
The reason is that the determinant of the matrix $Y$ may  vanish at a point $\zeta \in
D(\ul x (\xi), \wtilde R(\xi,\Sigma)^-) \setminus D(\ul x (\xi), \delta_{\cX}(\xi,U)^-)$,
while this cannot be the case in  $D(\ul x (\xi),  R(\xi,\Sigma)^-)$, otherwise the differential system for the wronskian $w := \det Y$, namely
\beq\label{eq:wronskian}
 \frac {\partial \, w}{\partial x_i} = ({\rm Tr} \, G_i) \,  w \;\; ,\ \forall\ i=1,\dots,d\,,
\eeq
would have a singularity in $U$, which is not the case.
Notice that $R(\xi,\Sigma)$ is then the maximum real number $r \leq 1$ such that the system $\Sigma$ admits a solution matrix $Y \in GL(d, \cO(D_{\cX}(\xi,r^-)))$.
\par
The advantage, and the intrinsic content,  of this definition may be better appreciated if we consider the category ${\bf MIC}(U/k)$ of  coherent $\cO_U$-modules with integrable connection $(\cE,\na)$,
$$\na : \cE \to \cE \otimes \Omega^1_{U/k} \; ,
$$
and the object
$(\cE := \cO_X^{\mu},\na)$ associated to $\Sigma$. If $\ul e = (e_1,\dots,e_{\mu})$ denotes the canonical basis of global sections of $\cO_X^{\mu}$, then, by convention,
\beq
\na(\ul e) = - \sum_{i=1}^d  (e_1 \otimes  dx_i ,\dots,e_{\mu} \otimes  dx_i  )  \, G_i \; ,
\eeq
so that $\Sigma$ is the differential system satisfied by the horizonal sections of $(\cE,\na)$.
\par
The abelian sheaf $\cE^{\na} = \ker (\na : \cE \to \cE \otimes \Omega^1_{U/k})$ for the $G$-topology of $U$, is not in general locally constant. If, on some analytic domain $V \subset U$,   $\cE^{\na}_{|V}$ is  locally constant, then it is necessarily a local system of $k$-vector spaces of rank $\mu$ on $V$ and the canonical
monomorphism
\beq
\label{eq:RHsheafintro}
\cE^{\na}   \otimes_k \cO_U \hookrightarrow  \cE \;  ,
\eeq
is in fact an {\it isomorphism}: this is the intrinsic content of our previous statement on the wronskian equation. Taking into account the fact that a locally constant sheaf of finite dimensional $k$-vector spaces on a disk is necessarily constant, we see that  $D = D_{\cX}(\xi,R(\xi,\Sigma)^-)$ is the maximal open disk centered at $\xi\in U$, and contained in $U$, where $(\cE,\na)$ is isomorphic to
the trivial connection $(\cO_D,d_D)^{\mu}$.
We may then give the
\begin{defn}[Alternative form] \label{def:altradius}
Let $(\cE,\na)$ be an object of
${\bf MIC}(U/k)$, with $\cE$ locally free of rank $\mu$ for the $G$-topology. For any $k$-rational point $\xi \in U$, we define the {\emph radius of convergence $R_{\cX}(\xi,(\cE,\na))$ of $(\cE,\na)$ at $\xi$} as the maximal open disk $D$ centered at $\xi$ and contained in $U$, such that  $(\cE,\na)_{|D}$ is isomorphic to
the trivial connection $(\cO_D,d_D)^{\mu}$.
\end{defn}
\par
Coming back to the explicit situation (\ref{eq:intrdiffsys}), there is a nice compact formula for the solution matrix $Y = Y_{\xi}$ of (\ref{eq:intrdiffsys}) at $\xi$, such that $Y_{\xi}(\xi) = I_{\mu}$. We write
\beq \label{eq:factorials}
\ul \a!=\prod_i\a_i! \; , \;  ( \ul x- \ul x (\xi))^{\ul \a} =  \prod_i (x_i - x_i(\xi))^{{\a}_i} \; ,  \;
\d^{\ul \a}
= \prod_i\frac{{\partial^{\a_i}}}{{\partial x_i^{\a_i}}} \; , \; \d^{[\ul \a]} = \frac{1}{\ul \a !} \d^{\ul \a} \; .
\eeq
By iteration of  the system (\ref{eq:intrdiffsys}) we obtain,
for any $\underline\a\in\N^d$, the equations
\beq\label{eq:stratintro}
\partial^{[\underline\a]} \, \vec y   = G_{[\underline\a]}\, \vec y \;\;\;\;\;\;{\rm ( \, resp.} \; \partial^{\underline\a} \, \vec y   = G_{\underline\a}\, \vec y \;{\rm )}\;,
\eeq
with $G_{[\underline\a]}$ and $G_{\underline\a} = \ul \a! G_{[\underline\a]} $, $\mu \times \mu$ matrices of functions analytic in $U$.  In particular, $G_{\ul 0} = I_{\mu}$ and $G_i $ is now written $G_{[\ul 1_i]} = G_{\ul 1_i}$, where $\ul 1_i = (0,\dots,0,1,0,\dots,0)$, with $1$ only at the $i$-th place.
\par
The $G_{\ul \a}$ satisfy the recursion relations
\beq \label{eq:stratrecintro}
G_{\ul \a +\ul 1_i} = \di(G_{\ul \a}) + G_{\ul \a} G_{\ul 1_i} \; .
\eeq
\par
The Taylor series of the fundamental solution matrix $Y_{\xi}$ of
(\ref{eq:intrdiffsys}) at $\xi \in U$   is
\beq \label{eq:taylorseriesintro}
 Y_{\xi} = \sum_{\underline\a \in \N^d} G_{[\ul \a]}(\xi) (\ul x -
\ul x(\xi))^{\underline\a} \in GL(\mu, \sH(\xi)[[\ul x - \ul x(\xi)]])\,,
\eeq
(for the $k$-rational point $\xi$, $\sH(\xi) =k$, of course) with radius
of convergence
\beq \label{eq:liminfintro}
\wtilde R(\xi,\Sigma) =  \liminf_{|\underline\a|_{\infty} \to
\infty} \l|G_{[\underline\a]}(\xi)\r|^{- {1 / |\underline\a|_{\infty}}} \in \R_{\geq 0}
\cup \{ \infty \} \; .
\eeq
\par

We now extend the previous definitions to all points $\xi \in U$.  The function $\xi \mapsto \wtilde R(\xi,\Sigma)$ will be defined in general by  formula (\ref{eq:liminfintro}). This amounts to the following  consideration on Berkovich analytic spaces.
 As explained in \cite[1.4]{BerkovichEtale}, we may consider the ground field extension of $U$ to $\sH(\xi)$, $U_{\sH(\xi)} = U \what{\otimes}_k \sH(\xi)$. It is a $\sH(\xi)$-analytic space equipped with a canonical compact projection map $\psi_\xi:  U_{\sH(\xi)} \to U$, and there is a canonical $\sH(\xi)$-rational  point $\xi^\p$ above $\xi$.  The system (\ref{eq:intrdiffsys}) may be viewed, with no change in notation, on
 $U_{\sH(\xi)} \to \A^d_{\sH(\xi)}$, where the field of constants is now $\sH(\xi)$, and formula (\ref{eq:liminfintro}) represents the radius of convergence of the fundamental solution matrix $Y_{\xi^\p}$ of (\ref{eq:intrdiffsys}), viewed on $U_{\sH(\xi)}$ at $\xi^\p$. We then define, for general $\xi \in U$, $D_{\cX}(\xi,r^\pm) := D_{\cX_{\sH(\xi)^{\circ}}}(\xi^\p,r^\pm)$, where $\cX_{\sH(\xi)^{\circ}} = \cX \times_{\kc} \Spf \sH(\xi)^{\circ}$, $\delta_{\cX}(\xi,U) := \delta_{\cX_{\sH(\xi)^{\circ}}}
 (\xi^\p,U_{\sH(\xi)})$,  and  $R(\xi,\Sigma) = \min (\wtilde R(\xi,\Sigma),\delta_{\cX}(\xi,U))$ (resp.  $R_{\cX}(\xi,(\cE,\na))= R_{\cX_{\sH(\xi)^{\circ}}}(\xi^\p,\psi_\xi^{\ast}(\cE,\na))$). We abusively call $D_{\cX}(\xi,r^\pm) $ the {\it open} (resp. {\it closed}) {\it disk of radius $r$ centered at $\xi \in X$}.

\par \medskip
In the situation (\ref{situation}), under the further condition that  {\it $U$ is a Laurent domain in $X$},  we prove that the function $\xi \mapsto  R(\xi,\Sigma)$ is upper semicontinuous on $U$, for its natural Berkovich topology. A preliminary fact, and this is where we need  $U$ to be a Laurent domain in $X$,  is that the function $\xi \mapsto  \delta_{\cX}(\xi,U)$ is  upper semicontinuous on $U$.  Moreover, if $U$ is the inverse image of a Laurent domain in $D^d_k(0,1^+)$, the function $\xi \mapsto  \delta_{\cX}(\xi,U)$  is continuous. If $U = X$ then $\xi \mapsto  R(\xi,\Sigma)$ is continuous at the maximal point $\eta_X$ of $X$. If $\dim X = 1$ and $\cX = \what \A^1_\kc$, we prove directly that  $\xi \mapsto  R(\xi,\Sigma)$ is continuous on $U$. Combining the last two results, we deduce that  $\xi \mapsto  R(\xi,\Sigma)$ is continuous if $\dim X = 1$ and $U$ is any affinoid neighborhood of $\eta_X$.
\par \medskip
Surprisingly enough, the simple statement above seems to be  new  even in the case when $U$ is the  closed unit disk $D_k(0,1^+)$ of dimension 1, a case extensively discussed in the literature (\cf \cite{DGS} and \cite{ChMeAST} for reference).
In the case of an ordinary differential system $\Sigma$ as (\ref{eq:intrdiffsys}) on an annulus
$$
U = C(r_1,r_2)=\l\{\xi:  r_1<|x(\xi)|<r_2\r\}
\subset \A^1_k \;
$$
with $0 < r_1$,
a simple convexity argument due to Christol and Dwork \cite{ChristolDwork}
shows that the function $\wtilde R$ is continuous when restricted to the  segment of points $(r_1,r_2) \to C(r_1,r_2)$, $r \mapsto t_r$,  where $t_r=t_{0,r}$ is the ``generic point at distance $r$ from $0$", \ie the point at the boundary of the disk $D_k(0,r^-)$.  They actually consider, precisely as we do,  the more invariant  function
\beq
\begin{array}{cccc}  R:&
(r_1,r_2)&\longrightarrow & \R_{\geq 0} \cup \{ \infty \}\\
\\
&r & \longmapsto &R(t_r, \Sigma)  :=  \min (r, \wtilde R(t_r, \Sigma) )\; \;\; .
\end{array}
\eeq
It is easy to show \cite[2.3]{ChristolDwork} that the function $\log r \mapsto \log \wtilde R(t_r, \Sigma)$ is  concave (\ie $\cap$-shaped), hence continuous, in $(r_1,r_2)$. The function $r \mapsto \wtilde R(t_r, \Sigma)$ is therefore continuous on $(r_1,r_2)$.
In this situation, the system is said to be {\it solvable at} $r_2$ if the $\displaystyle \lim_{r \to r_2^-}  R(t_r,\Sigma)$, which certainly exists, is  $= r_2$ (and similarly for $r_1$). Systems solvable at $r_2$ (resp. $r_1$) are only understood on $C(r_2 - \veps, r_2)$ (resp. $C(r_1 , r_1 + \veps)$), for small values of $\veps >0$, by the theory of {\it factorization according to the slopes} due to Christol and Mebkhout \cite{ChMe3} \cite{ChMe4}.
 In the special case of a {\it Robba system} \cite[3.1]{ChMe2}, \ie of a system $\Sigma$ on $C(r_1,r_2)$, such that $ R(t_r,\Sigma) = r$ for every $r \in (r_1,r_2)$, it follows from Dwork transfer theorems \cite[IV.5.2]{DGS}, that $ R(\xi,\Sigma) = |x(\xi)|$, for every $\xi \in C(r_1,r_2)$.\footnote{If for two values $R_1$ and $R_2$, with $r_1 <R_1 < R_2 <r_2$,
 $ R(t_{R_i},\Sigma) =  R_i$, $i =1,2$, then
 $ R(t_r,\Sigma) =  r$, for all $r \in [R_1,R_2]$ \cite[Cor. in App. I]{DGS}.}
 This simplest case is of high interest, even (or maybe {\it especially}) when its features depend on the existence of a strong Frobenius structure. A notion of {\it exponents} is then available, and under an arithmetic condition on them (automatic in case of a strong Frobenius structure) the system admits a Fuchs-type decomposition over $C(r_1,r_2)$ \cite{ChMe2} \cite{DworkEXP}. {\it Our paper deals with the deviation of a system from being of Robba type}.
 \par  \medskip
We prove a far-reaching generalization of the  Dwork-Robba theorem
\cite[IV.3.1]{DGS} on effective bounds  for the growth of local
solutions (theorem (\ref{thm:DworkRobba}) and its corollaries). The only difference from the version published by Gachet \cite{gachet} is the formulation
on Berkovich analytic spaces, which is however crucial  in
our proof of the upper semicontinuity of the radius of convergence
(\cf \S\ref{subsection:USC}).

\par \medskip
It turns out that Berkovich
analytic spaces represent an ideal framework for the study of
$p$-adic differential equations. They contain the {\it generic points} in the sense of Dwork and Robba, as honest points. This gives  great flexibility to the ``rigid" geometry setting and permits in the end to
generalize classical one-dimensional
results of Dwork, Robba and Christol to analytic spaces.

\medskip

\par

\medskip
{\bf Acknowledgement.}
We are indebted to Vladimir Berkovich  for considerable help both in the formulation
and in the proof of the theorems appearing in this paper. We are
also grateful to Yves Andr\'e and Kiran Kedlaya for showing so much
interest in our results.

\section{Generalities and notation} \label{notation}
\par
We refer to our assumptions (\ref{situation}). For any subset $S \subset X$, the notation $||~||_S$ will refer to the supnorm on $S$. For example, for an analytic domain $V \subset X$, let $\sA_V^+$, denote the $k$-Banach algebra of {\it bounded} analytic functions on $V$, equipped with  $k$-Banach norm $||~||_V$. If $V = \cM(\sA_V)$ is affinoid, then $\sA_V^+ = \sA_V$, and
the $k$-affinoid algebra $\sA_V$ will  be viewed as a $k$-Banach algebra via $||~||_V$. We will denote by $\cL_k(\sA^+_V)$ the $k$-vector space of $||~||_V$-bounded $k$-linear endomorphisms of $\sA^+_V$, equipped with the corresponding operator norm $|~|_V := ||~||_{V,\op}$. Since  $\forall \varphi$,$\psi \in \cL_k(\sA^+_V) $,
$||\psi \circ \varphi ||_{V,\op} \leq  ||\psi ||_{V,\op}||\varphi ||_{V,\op}$,
$\cL_k(\sA^+_V) = (\cL_k(\sA^+_V), |~|_V)$ is a $k$-Banach algebra.
\par
For a matrix $G = (g_{ij})$ of elements in a $k$-Banach algebra $(\sB,||~||)$, we will set
$$||G|| := \sup_{i,j} ||g_{ij} || \; .$$
Then $||G \cdot H|| \leq ||G|| \cdot ||H||$, whenever multiplication of matrices makes sense, and the $k$-algebra $\cM_{n \times n}(\sB)$ of $n \times n$-matrices with entries in $\sB$, equipped with the norm $||~||$, is a Banach $k$-algebra.
\par

\subsection{Ground extension functor and continuity} \label{subsec:ground}
We need a definition extracted from \cite[1.4]{BerkovichEtale}. Let $L$ be any complete valued field extension of $k$; the {\it ground extension functor} associates to any $k$-analytic space $Y$ an $L$-analytic space $Y_L = Y \what\otimes_k L$ equipped with a canonical projection $\psi_{Y,L} = \psi_{Y,L/k} :Y_L \to Y$. In the case of a $k$-affinoid space $Y = \sM(\sA)$, $Y_L = \sM(\sA_L)$, where $\sA_L$ is the $L$-affinoid algebra $\sA \what\otimes_k L$, and the map $\psi_{Y,L}$ corresponds to the inclusion $\sA \to \sA_L$, $a \mapsto a \otimes 1$.  By construction, the map $\psi_{Y,L}$ is compact, \ie for any compact subset $C$ of $Y$, $\psi_{Y,L}^{-1}(C)$ is compact. We will be dealing with a family $\sF$ of functions defined on the analytic spaces over $k$ in a class  $\sS = \bigcup_L \sS_L$, where  $\sS_L$ is a class of $L$-analytic spaces and $L$ varies over  completely valued  field extensionof $k$.  We assume that $\sF = \bigcup_L \sF_L$, $\sF_L =
\{ \varphi_Y:Y \to S\}_{Y \in \sS_L}$, all functions taking values in  a fixed  topological space $S$. We will assume that the class $\sS$ is stable by ground extensions,
and that  the family $\sF$ is  {\it compatible with base change}, in the sense that if $Y \in \sS_L$, $\varphi_Y \in \sF_L$, and $L^\p/L$ is a completely valued extension, $Y_{L^\p} \in \sS_{L^\p}$, and $\varphi_{Y_{L^\p}} = \varphi_Y \circ \psi_{Y,L^\p/L} \in \sF_{L^\p}$.
The following general lemma shows  that, to prove continuity of the functions in $\sF_k$, no loss of generality is involved in assuming that the base field $k$ is maximally complete and algebraically closed.
\begin{lemma}\label{lemma:fieldext} Let $Y$ be any $k$-analytic space, $L$ be a complete valued field extension of $k$ and $Y_L = Y \what\otimes_k L$ be the extension of $Y$ over $L$. Then the natural topology of $Y$ is the quotient topology of the projection map  $\psi_L = \psi_{Y,L}:Y_L \to Y$.
\end{lemma}
\begin{proof} We first prove that the map $\psi_L$ is closed. Let $C$ be a closed subset of $Y_L$. Let $y$ be a point of $Y \setminus \psi_L(C)$, and let $D_2$ be a compact neighborhood of $y$ in $Y$. Then $D_1 = \psi_L^{-1}(D_2)$ is a compact subset of $Y_L$. The intersection $C \cap D_1$ is then compact; its image $\psi_L(C \cap D_1)$ is then closed, so that $D_2 \setminus \psi_L(C \cap D_1)$ is a neighborhood of $y$ in $Y$ not intersecting $\psi_L(C)$.
\par \noindent The conclusion follows from  \cite[2.4]{GenTop}.
\end{proof}

 It follows from the previous lemma that a function on $Y$ is continuous if
and only if its lift to $Y_L$ is continuous. In particular,
\begin{cor}
The functions in $\sF_k$ are continuous if  there exists a completely valued extension field $L/k$ such that all functions in $\sF_L$ are continuous.
\end{cor}
This will allow us to assume in certain cases, without loss of generality, that  the ground field $k$ is maximally complete and algebraically closed.
\par
\medskip
We recall here for completeness that a function $\varphi:T \to \R$, where $T$ is
any topological space is {\it upper semicontinuous} or {\it USC}
(resp. {\it lower semicontinuous}  or {\it LSC}) if $\forall t_0 \in
T$ and $\veps >0$, there exists a neighborhood $U_{t_0,\veps}$ of
$t_0$ in $T$ such that
$$
\varphi(t) <  \varphi(t_0)  + \veps \;\; {\rm ( resp.}\;\; \varphi(t) >  \varphi(t_0)  - \veps \; {\text )}\;
$$
$\forall t \in U_{t_0,\veps}$. If $\forall \a \in I$, $\varphi_{\a}$
is USC (resp. LSC), then
$$
\varphi = \inf_{\a \in I} \varphi_{\a} \;\; {\rm ( resp.}\;\; \varphi = \sup_{\a \in I} \varphi_{\a}  \; {\text )}\;
$$
is USC (resp. LSC).
\par \noindent

\section{(Semi-) continuity of formal invariants}

\subsection{Diameter} \label{subsec:diambound}

We work here under  the assumptions of (\ref{situation}). We recall that we have defined open (resp. closed) disks $D_{\cX}(\xi,r^\pm)$ centered at $\xi \in X$ of radius $r \in (0,1]$ (resp. $r \in (0,1)$). We say that a disk $D_{\cX}(\xi,r^\pm)$ is $k$-rational if its center $\xi$ may be chosen in $X(k)$.

\begin{prop}
For any $\xi \in U$,  the diameter $\delta_{\cX}(\xi,U) > 0$.
\end{prop}
\begin{proof} We follow the notation of \cite[2.5]{Berkovich}. We may assume that $\xi$ is a $k$-rational point of $U$, and that $U = \cM (\sA_U)$ is an affinoid  contained in a disk $D = D_{\cX}(\xi,r^+)$, with $r \in |k^{\times}|$. Then
$D$ is isomorphic as a $k$-analytic space to $\cM( k\{r^{-1}X \})$, with  $k\{r^{-1}X \} = k\{r^{-1}X_1,\dots,r^{-1}X_d \}$, so we regard $U$ as an affinoid in $D_k^d(0,r^+)$. Let $\chi_{\xi}: \sA_U \to k$ (resp. $\chi^\p_{\xi}: k\{r^{-1}X\} \to k$) be the bounded character corresponding to $\xi \in U$ (resp. $\xi \in D$). Notice that $\chi_{\xi}$ may be viewed as a bounded $k\{r^{-1}X \}$-homomorphism $ \sA_U \to k$.
The reduced character $\wtilde \chi_{\xi}: \wtilde \sA_U \to \wtilde k$ obviously satisfies condition $(d)$ of
\cite[2.5.2]{Berkovich} \ie the ring $\wtilde \chi_{\xi} (\wtilde \sA_U) = \wtilde k$ is integral over $\wtilde \chi^\p_{\xi} (\wtilde {k\{r^{-1}X \}}) = \wtilde k$, hence $\chi_{\xi}$ is inner with respect to $k\{r^{-1}X \}$.
In that case, it is known that $\xi$ lies in the topological interior of $U$ in $D$ \cite[2.5.13]{Berkovich}. So, $U$ hence contains a non trivial disk $D_{\cX}(\xi,\veps^-)$ centered at $\xi$ and $\delta_{\cX}(\xi,U) \geq \veps > 0$.
\end{proof}
%

\par
 \medskip
We recall that a  \emph{Laurent (affinoid) domain} in any $k$-affinoid space  $Y = \cM(\sB)$ is a
domain of the form
\beq\label{eq:laurent}
Y(r^{-1}f,sg^{-1}) =
\{x \in Y \mid |f_{i}(x)| \leq r_i, \,|g_j(x)| \geq s_j,\,1 \leq i
\leq n,\, 1 \leq j \leq m \}
\eeq
where $f_i$, $g_j \in \sB$, and
$r_i $, $s_j $ are positive real numbers.
\par

\subsection{Trivial estimate} \label{sec:trivialestimate}
We will assume here that the entries of the matrices $G_i$ in (\ref{eq:intrdiffsys}) are {\it bounded} analytic functions on $U$, \ie elements of the $k$-Banach algebra $\sA^+_U$. For any analytic domain $V \subset X = \cX_{\eta}$, the derivations
$\di$, for $i=1,\dots,d$, are bounded $k$-linear operators on  $\sA^+_V$. Let $|\di|_V$ denote the operator norm of $\di$ on the $k$-Banach algebra $\cL_k(\sA^+_V)$.
Then,
\begin{prop}\label{prop:trivialestimate}
For any $\xi \in U$ we have:
\beq
\wtilde R(\xi, \Sigma)  \geq
\frac{|p|^{\frac{1}{p-1}}}
{\max_{i=1,\dots,d}\l(|\di|_U,||G_i||_U\r)}  >0\,.
\eeq
\end{prop}

\begin{proof} It follows from (\ref{eq:stratrecintro}) that
for any $\ul\a\in\N^d$, with $\a_i>0$, we have
$$
\begin{array}{rcl}
||G_{\ul\a}||_U
&\leq&\ds\sup\l(\l|\l|\frac{\partial G_{\ul\a-\ul 1_i}}{\partial x_i}\r|\r|_U,
  \l|\l|G_{\ul\a-\ul 1_i}G_i \r|\r|_U \r)\\ \\
&\leq &\ds\l| \l|G_{\ul\a-\ul 1_i}\r|\r|_U \sup\l(|\di|_U,||G_i||_U
\r)\,.
\end{array}
$$
Recursively we obtain
$$
||G_{\ul\a}||_U
\leq \sup_{i=1,\dots,d}\l(|\di|_U,||G_i||_U\r)^{|\ul\a|_\infty}\, ,
$$
hence
$$
||G_{[\ul\a]}||_U
\leq \sup_{i=1,\dots,d}
\l(|\di|_U,||G_i||_U \r)^{|\ul \a|_\infty} / ||\ul\a|_\infty!|_p
\; .
$$
Finally
$$
||G_{[\ul\a]}||_U^{1/|\ul \a|_\infty}
\leq \sup_{i=1,\dots,d}
\l(|\di|_U,||G_i||_U \r) / |p|^{\frac
{|\ul \a|_\infty - S_p(|\ul \a|_\infty)}
{(p-1)|\ul \a|_\infty}
}
\; ,
$$
where $S_p(n) \leq \log_p n$ is the sum of $p$-adic digits of the natural number $n$, from which we deduce the formula in the statement.
\end{proof}

\subsection{(Upper semi-)continuity of $\xi \mapsto \delta_\cX(\xi,U)$ for $U$ a Laurent domain in $X$}
Let
$U = X(r^{-1}f,sg^{-1})$  be a Laurent domain  in $X$ (\ref{eq:laurent}), so that  $f_i$, $g_j \in
\sA$ and
$r_i $, $s_j $ are positive real numbers. We will say that $U= X(r^{-1}f,sg^{-1})$ is a {\it special} Laurent domain in $X$, if $f_i$, $g_j \in
k \{x_1,\dots , x_d\}$.
Since
$$
\ds X(r^{-1}f,sg^{-1})=\l(\mathop\cap_i X(r_i^{-1}f_i)\r)\cap
\l(\mathop\cap_j X(s_jg_j^{-1})\r)
$$
we actually have
$$
\delta_\cX(\xi, X(r^{-1}f,sg^{-1})) =
\min_{i,j}\Big(\delta_\cX\big(\xi, X(r_i^{-1}f_i)\big),
\delta_\cX\big(\xi, X(s_jg_j^{-1})\big)\Big)\,.
$$

\begin{prop}
{\rm(\cf \cite{BerkovichLetter})}
\label{prop:diamlaurent}
Let $f,g \in  \sA$, and let
$U = X(r^{-1}f)$, with $r >0$   (resp.
$U = X(sg^{-1})$, with $s > 0$), and let $\xi \in U$. Then:
\beq
\label{eq:diamlaurent1}
\delta_\cX(\xi,U) = \min (1, \inf_{1 \leq
|\underline\a|_{\infty},f^{[\underline \a]}(\xi)\neq 0} \{
r^{1/|\underline\a|_{\infty}} |f^{[\underline
\a]}(\xi)|^{-1/|\underline\a|_{\infty}} \} )\;.
\eeq
(resp.
\beq
\label{eq:diamlaurent2}
\delta_\cX(\xi,U) = \min (1, \inf_{1 \leq
|\underline\a|_{\infty},g^{[\ul\a]}(\xi)\neq 0} \{
|g(\xi)|^{1/|\ul\a|_\infty}
|g^{[\ul\a]}(\xi)|^{-1/|\ul\a|_\infty}\}) \;. \hbox{ )}
\eeq
In particular, for any Laurent domain $U \subset X$, the function $\xi \mapsto \delta_\cX(\xi,U)$ is upper semicontinuous on $U$.
If  $f,g \in k\{x_1,\dots , x_d\}$,
the infima (\ref{eq:diamlaurent1}), (\ref{eq:diamlaurent2}) are realized on a finite set of $\ul\a\in\N^d$, depending only upon $U$.
In particular, if $U$ is a special Laurent domain in $X$, the function $\xi \mapsto \delta_\cX(\xi,U)$ is continuous on $U$.
\end{prop}

\begin{proof}
We consider the case of $U = X(r^{-1}f)$, for $r \in (0,1)$ first.
We extend the base field to $\sH (\xi)$, so that the canonical
point $\xi^{\prime}$ over $\xi$ has a  neighborhood which is a
disk centered at $\xi^{\prime}$. We set ${\ul \xi} = (\xi_1,\dots,\xi_d) = (x_1(\xi),\dots,x_d(\xi))$.
The Taylor expansion at $\xi$, $g \mapsto \sum_{\ul \a} g^{[\ul \a]}(\xi) (\ul X - \ul \xi)^{\ul \a}$, produces an isometric embedding
\beq
T_{\xi,\ul x}: \sA \to k \otimes \kc\{x_1-\xi_1,\dots,x_d-\xi_d\} \; ,
\eeq
of $\sA$, equipped with the supnorm on $X$, into the ring of bounded analytic functions on $D^d_k(\ul \xi,1^-)$, with the natural norm.  The diameter
$\delta_\cX(\xi,{U})$ is then characterized as follows
$$
\delta_\cX(\xi,{U})=\sup \l\{\veps \in (0,1) \; : \;|f(x)|\leq
r\; \forall x \in D_{\sH (\xi)}^d((\xi_1,\dots.\xi_d),\veps^+)\r\}\, .
$$
Since
\beq\label{eq:disuguaglianzadiametro} \sup_{x \in
D_{\sH (\xi)}^d((\xi_1,\dots.\xi_d),\veps^+)}|f(x)|
=\sup_{\ul\a\in\N^d}|f^{[\ul\a]}(\xi)|\veps^{|\ul\a|_\infty} \leq r
\eeq
we deduce that
\beq \label{eq:deltainf}
\delta_\cX(\xi,{U}) = \min(1,  \inf_{1 \leq |\underline\a|_{\infty} ,\,f^{[\underline \a]}(\xi)\neq 0} \{
r^{1/|\underline\a|_{\infty}}|f^{[\underline
\a]}(\xi)|^{-1/|\underline\a|_{\infty}} \}) \;,
\eeq
and hence that $\xi \mapsto \delta_\cX(\xi,{U})$ is an upper semi-continuous function of
$\xi \in U$.
\par
If we now assume that  $f \in  k \{x_1,\dots , x_d\}$, then
$\lim\limits_{|\underline\a|_{\infty} \to \infty}||f^{[\underline\a]}||_X = 0$.
Then
 there exists a natural number $N$ such that
$||f^{[\underline\a]}||_X < r$, $\forall x \in X$, as soon as
$|\underline\a|_{\infty} \geq N$. The infimum in  (\ref{eq:deltainf}) is then really a minimum on the finite set $|\underline\a|_{\infty} <  N$. The function $\xi \mapsto \delta_\cX(\xi,{U})$ is continuous in this case.
\par
We now  consider the case of $U = X(sg^{-1})$, $g \in \sA$.
As in the previous case, we extend our spaces to $\sH(\xi)$, so that
we have the canonical point $\xi^{\prime}$ over $\xi$ with
$|g(\xi^\p)| = |g(\xi)|  \geq s$. Suppose that there exists $\omega \in
D_{\sH (\xi)}^d((\xi_1,\dots.\xi_d),\veps^+)$, for some $\veps \in (0,1)$, such that
$|g(\omega)|< |g(\xi)|$.
We deduce from Corollary \ref{cor:Robba} in the appendix, that $g$
has a zero in the  disk $D_{\sH (\xi)}^d((\xi_1,\dots.\xi_d),\veps^+)$ so that
$\veps > \delta_\cX(\xi,U)$.
\par
In other words, we have proven that $\delta_\cX(\xi,U)$ is
precisely the minimum distance of a zero of $g$ from $\xi^\p$. We
use Robba's theory of Newton polygons (\cf corollary \ref{cor:Robba}  in the appendix)
to obtain an explicit
formula. The conclusion is that
$$
\delta_{\cX}(\xi,{U}) =  \min(1, \inf_{1 \leq |\ul\a|_{\infty}} \{
|g(\xi)|^{1/|\ul\a|_{\infty}}
|g^{[\ul\a]}(\xi)|^{-1/|\ul\a|_{\infty}} \}) \;.
$$
As in the previous case, the infimum is really a minimum, and if $g \in  k \{x_1,\dots , x_d\}$ it is a minimum in a finite
set of $\ul\a$'s. We conclude as in the previous case.
\end{proof}
\section{The Dwork-Robba theorem and the upper semicontinuity of
$\xi \mapsto  R(\xi,\Sigma)$}

\subsection{The global growth estimate}

 We set ourselves in the situation of  (\ref{eq:intrdiffsys}). We will need the following estimate, a corollary of the generalized form  of the theorem of Dwork and Robba \cite[Chap. IV, Thm. 3.1]{DGS} given below.

 \begin{thm}{\bf (Growth estimate)} \label{thm:DworkRobbaBerk}
 Assume the entries of the matrices $G_i$ in (\ref{eq:intrdiffsys}) are bounded analytic functions on the analytic domain $U$. For any  $\xi \in U$  let $R(\xi) = R(\xi,\Sigma)$ be the radius of convergence of $\Sigma$ at $\xi$. Let, for any $\ul\be\in\N^d$, $C_{\ul\be}= C_{\ul\be}(\Sigma, U)$ be the constant
 \beq C_{\ul\be} =   \l| \l| \, {G}_{\ul\be}  \r| \r|_U  = \sup_{\zeta \in U}
\l| \ul \be ! \, {G}_{[\ul\be]} (\zeta) \r|    \; ,
\eeq
and
 $C= C(\Sigma, U)$ be
\beq C = \max_{|\ul\be|_\infty\lneqq\mu} C_{\ul\be}  \; .
\eeq
For any $\ul\a\in\N^d$ we have the
following growth estimate on the coefficients of $Y_{\xi}$
\beq\label{eq:DworkRobbaBerk}
|G_{[\ul\a]} (\xi)| \leq \l(  \sum_{|\ul\be|_\infty\lneqq\mu} C_{\ul\be} R(\xi)^{|\ul\be|_\infty} \r) \{|\ul\a|_\infty,(\mu -1)\}_p
R(\xi)^{- |\ul\a|_\infty} \leq C\{|\ul\a|_\infty,(\mu -1)\}_p
R(\xi)^{- |\ul\a|_\infty}\ ,
\eeq
where
$$
\{s,n\}_p=\sup_{1\leq\la_1\lneqq\la_2\lneqq\cdots\lneqq\la_n\leq s}
\l({1\over {|\la_1\cdots\la_n|_p}}\r)\ .
$$
\end{thm}
\begin{rmk} $\{s,n\}_p\leq s^n$.
\end{rmk}
\begin{rmk} In practice, the estimate (\ref{eq:DworkRobbaBerk}) is used in the form
\beq \label{eq:pracDR}
|G_{[\ul\a]} (\xi)| \leq C |\ul\a|_\infty^{\mu -1}
R(\xi)^{- |\ul\a|_\infty} \; .
\eeq
\end{rmk}
\begin{cor} \label{uniform} For any $\veps >0$, there exists $s_{\veps} \in \N$, such that for every $\ul \a \in \N^d$, with $|\ul\a|_\infty \geq s_{\veps}$ and every $\xi \in U$
\beq
|G_{[\ul\a]} (\xi)| ^{1/|\ul\a|_\infty} \leq (1 + \veps)/
R(\xi) \; .
\eeq
\end{cor}
 We mention a variation of (\ref{thm:DworkRobbaBerk}), which is often useful. Let ${\rm cl}_X(U)$ be the closure of $U$ in $X$, and let $\sA_{(U)}$ be the localization of the algebra $\sA$ with respect to the elements which do not vanish on ${\rm cl}_X(U)$. Let $\sH(U) \subset \sA_U^+$ denote the completion of $\sA_{(U)}$ in the supnorm $||~||_U$. The elements of $\sH(U)$ will be called {\it analytic elements on $U$}; they  define continuous real valued functions on  ${\rm cl}_X(U)$. Namely, if $h \in \sH(U)$ is the uniform limit $\displaystyle h = \lim_i R_i/S_i$, where $R_i,S_i \in \sA$, and $S_i$ does not vanish on ${\rm cl}_X(U)$, we may define for any limit  point $\xi \in {\rm cl}_X(U)$, $\displaystyle \xi = \lim_{j \to \infty} \eta_j$, $\eta_j \in U$, $|h(\xi)| = \displaystyle  \lim_i R_i(\eta_j)/S_i(\eta_j)$.
 Assume the entries of the matrices $G_i$ in (\ref{eq:intrdiffsys}) are analytic elements on the analytic domain $U$, and that the function $\xi \mapsto \delta_{\cX}(\xi,U)$ admits a continuous extension on
 ${\rm cl}_X(U)$. Then, $|G_{[\ul\a]} (\xi)| $ exists $\forall \, \ul \a$ and $\forall \, \xi \in {\rm cl}_X(U)$, and $\wtilde R(\xi,\Sigma)$ is defined by formula (\ref{eq:locliminf}) $\forall \, \xi \in {\rm cl}_X(U)$. Let us define
 $R(\xi,\Sigma)$, $\forall \, \xi \in {\rm cl}_X(U)$, by formula (\ref{eq:radiusdef}).

Then we have
  \begin{thm} \label{thm:DRanelem}
 Assume the entries of the matrices $G_i$ in (\ref{eq:intrdiffsys}) are analytic elements on the analytic domain $U$, and that the function $\xi \mapsto \delta_{\cX}(\xi,U)$ admits a continuous extension on
 ${\rm cl}_X(U)$.  Let, for any $\ul\be\in\N^d$, $C_{\ul\be}= C_{\ul\be}(\Sigma, U)$ and
 $C= C(\Sigma, U)$ be
the constants defined in (\ref{thm:DworkRobbaBerk}).
For any $\ul\a\in\N^d$ we have
again the growth estimate (\ref{eq:DworkRobbaBerk}).
\end{thm}
\begin{cor} Remark (\ref{eq:pracDR}) and corollary (\ref{uniform}) hold under the assumptions of theorem (\ref{thm:DRanelem}).
\end{cor}
 \subsection{The generalized Dwork-Robba theorem}

 The following discussion, due to Dwork,  has been previously made available by Gachet \cite{gachet}.  We rediscuss it here in the framework of Berkovich spaces. We  set ourselves in a slightly more general situation, namely we assume that the matrices $G_i$ of the system $\Sigma$ in (\ref{eq:intrdiffsys}), are meromorphic functions in a polydisk $D(a,r^-)$, for $a =
(a_1,\dots,a_d) \in k^d$, $r =
(r_1,\dots,r_d) \in \R_{>0}^d$
 $$D(a,r^-) = D^d_k(a,r^-) = \{ \xi \in \A_k^d \, | \, |x_i(\xi) - a_i| < r_i \; , \; \forall i = 1, \dots, d \; \} \; .
 $$
 The {\it field $\cM(D(a,r^-))$ of meromorphic functions on
$D(a,r^-)$} is defined as the quotient field of the
integral domain $\cO(D(a,r^-))$.
For any
$\rho = (\rho_1, \dots,\rho_d)$, $0 < \rho_i <r_i$, the maximal
point $t_{a,\rho}$ of $D(a,\rho^+)$ belongs to  $D(a,r^-)$ and
defines a multiplicative map $\cM(D(a,r^-)) \to \sH(t_{a,\rho})$, $f
\mapsto f(t_{a,\rho})$. For $f \in \cM(D(a,r^-))$, the function
$\rho \mapsto f(t_{a,\rho})$ is continuous, as shown in the
appendix, but not necessarily bounded for $0 < \rho_i <r_i$. We
define the {\it boundary seminorm} $||~~||_{a,r}$ on $\cO(D(a,r^-))$
as
\beq
||f||_{a,r} = \limsup_{\rho \to r}|f(t_{a,\rho})|\, \in
\R_{\geq 0} \cup \{\infty\} \;\; ,\;\;  f \in \cM(D(a,r^-)) \; .
\eeq
It is clear that
$$ ||f+g||_{a,r} \leq \sup (||f||_{a,r} ,||g||_{a,r}) \; ,
$$
for all $f,g \in \cM(D(a,r^-)) $, and that
$$ ||f \, g||_{a,r} \leq ||f||_{a,r} \, ||g||_{a,r} \; ,
$$
whenever the right side is defined (the only case excluded is  $||f||_{a,r}=0$, $||g||_{a,r} = \infty$). Notice that, in one variable $X$, $||1 / \log (1- X)||_{0,1} =0$. If
$f = \sum_{\ul \a \in \Z^d} a_{\ul \a}(\ul X - \ul a)^{\ul \a}$, then
\beq
||f||_{a,r} = \sup_{\ul \a} | a_{\ul \a} |
 \; .
\eeq

 We have the following generalization of the theorem of Dwork and Robba \cite[Chap. IV, Thm. 3.1]{DGS}.

\begin{thm}\label{thm:DworkRobba}
Suppose that the system (\ref{eq:intrdiffsys}) has meromorphic coefficients on $D(a,r^-) \subset \A^d_k$, $a =
(a_1,\dots,a_d) \in k^d$, and that it
admits a solution matrix $Y \in GL(\mu, \cM(D(a,R^-)) )$, meromorphic in $D(a,R^-)  = D(a,(R_1,\dots,R_d)^-) \subset D(a,r^-)$. Then, for any $\ul\a\in\N^d$ we have the
following estimate
\beq\label{eq:DworkRobba}
\l|\l|{G}_{[\ul\a]}\r|\r|_{a,R} \leq C\{|\ul\a|_\infty,(\mu -1)\}_p
\ul R^{-\ul\a}\ ,
\eeq
where $\ul R^{-\ul\a}=R_1^{-\a_1}\cdots R_d^{-\a_d}$,
$$
C=\max_{|\ul\be|_\infty\lneqq\mu}
\l(\ul R^{\ul\be}
\l|\l| \ul \be ! \, {G}_{[\ul\be]}\r|\r|_{a,R} \r)\ ,
$$
$||~~||_{a,R}$ denotes the boundary seminorm on $\cM(D_k^d(a,(R_1,\dots,R_d)^-))$, and
$$
\{s,n\}_p=\sup_{1\leq\la_1\lneqq\la_2\lneqq\cdots\lneqq\la_n\leq s}
\l({1\over {|\la_1\cdots\la_n|_p}}\r)\ .
$$
\end{thm}

\begin{proof} We may assume that $a = 0$, which will simplify notations.
Let us consider the completion $\sK_{\ul b ,R}$ of the field $k(\ul b) = k(b_1,\dots,b_d)$
(of rational functions
  in the variables $b_1,\dots,b_d$)
with respect to the absolute value $ |~|_{\ul b ,R} : f \mapsto |f(t_{0,R})|$, with respect to the variables $\ul b$, so that
$|b_i|_{\ul b, R} = R_i$, for any $i$, and $|c|_{\ul b, R} = |c|$, for any $c \in k$. We have an injective
map of $k$-algebras
\beq \label{genline}
\begin{array}{ccc}
\ds\cM(D_k^d(0,(R_1,\dots,R_d)^-))
             & \longrightarrow & \cM(D_{\sK_{\ul b ,R}}(0,1^-))
             \\ \\
f(X_1, \dots, X_d) & \longmapsto & f(b_1 Z, \dots, b_dZ)
\end{array}\ .
\eeq

For any $\ul \a \in \N^d$, we will shorten $\ul \a ! \, {G}_{[\ul\a]}$ into ${G}_{\ul\a}$, so that (\ref{eq:stratintro}) becomes
\beq
\partial^{\ul \a} \vec y = {G}_{\ul\a} \vec y \; .
\eeq
We
denote by $\widetilde{G}_{\ul\a}(Z \ul b)$ the image of
${G}_{\ul\a}$,
via the injective morphism (\ref{genline}), and define for any $l\in\N$
\beq \label{genH}
{\cH}_{[l]}(Z)={1\over l!}{\cH}_l(Z)=
{1\over l!}\l(\sum_{|\ul\a|_\infty=l}
\widetilde{G}_{\ul\a}( Z \ul b ){\ul b}^{\ul\a}\r)\,,
\eeq

We reduce to the case of dimension 1 \emph{via} a generic line argument:

\begin{lemma} Consider the system of ordinary differential equations
\beq \frac{d}{dZ} \vec y = {\cH}_1(Z) \vec y \; ,
\eeq
where ${\cH}_1(Z)$ is the matrix of meromorphic functions on $D_{\sK_{\ul b ,R}}(0,1^-)$ appearing in
(\ref{genH}). Then, in the notation (\ref{genH})
\beq
\l(\frac{d}{dZ}\r)^l \vec y = {\cH}_l(Z) \vec y \; ,
\eeq
\end{lemma}

\begin{proof}
It is enough to prove that the matrices $\cH_l(Z)$
verify the recursive relations induced by the Leibnitz formula, namely:
$$
\begin{array}{rcl}
\frac{d}{d Z}{\cH}_l(Z)+{\cH}_l(Z){\cH}_1(Z)
  &=&\ds \sum_{|\ul\a|_\infty=l\atop i=1,\dots,d}
      \l(\l(\frac{\partial}{\partial X_i}\widetilde{G}_{\ul\a}\r)( Z \ul b)+
      \widetilde{G}_{\ul\a}(Z \ul b)\widetilde{G}_{\ul 1_i}(Z \ul b)\r)
      \ul b^{\ul\a+\ul 1_i}\\ \\
  &=&\ds \sum_{|\ul\a|_\infty=l+1}
      \widetilde{G}_{\ul\a}(Z \ul b)\ul b^{\ul\a}={\cH}_{l+1}(Z)\ .
\end{array}
$$
\end{proof}

We can now conclude the proof of the theorem. We denote by $||~~||_{\ul b ,R}$ the boundary seminorm on
$\cM(D_{\sK_{\ul b ,R}}(0, 1^-))$, defined at the beginning of this section, relative to the complete field $\sK_{\ul b ,R}$. We  on the other hand keep denoting $||~~||_{0,R}$ the boundary seminorm on
$\cM(D_k^d(0, (R_1, \dots,R_d)^-))$.
We have
$$ ||\cH _{[l]}||_{\ul b ,R} =
\sup_{|\ul \beta|_\infty = l}\l|
{\frac{1}{l!} {G}_{\ul \beta}}\r|_{0,R} \ul R^{\ul \be}\,.
$$
The classical theorem of Dwork-Robba in the one variable case (\cf [DGS, IV.3.2])
implies that for any $l = |\ul \a|_{\infty}$ we obtain the estimate
$$
\begin{array}{rcl}
||{G}_{[\ul \alpha]}||_{0,R} \ul R^{\ul \a}
&\leq&||\cH _{[l]}||_{\ul b, R} \\ \\
&\leq&\ds\{l,\mu-1\}_p
    \sup_{j\leq\mu-1}||\cH _j||_{\ul b ,R}\\ \\
&\leq&\ds\{|\ul\a|_\infty,\mu-1\}_p
    \sup_{|\ul \beta|_\infty\leq\mu-1}(\ul R^{\ul \beta}
    ||{G}_{\ul\beta}||_{0,R})\\ \\
&\leq& C \{ |\ul\a|_\infty,\mu-1 \}_p\,.
\end{array}
$$
This ends the proof.
\end{proof}

\begin{cor}\label{cor:DworkRobba1} Suppose the matrices $G_i$ are holomorphic and bounded in $D(a,R^-)$. Let
$$
C=\max_{|\ul\be|_\infty\lneqq\mu}
\l(\ul R^{\ul\be}
\l|\l|{G}_{\ul\be}\r|\r|_{D(a,R^-)} \r)\  .
$$
Then, for any $\ul\a\in\N^d$ we have the
following estimate
\beq\label{eq:DworkRobba1}
|{G}_{[\ul\a]}(a)| \leq C\{|\ul\a|_\infty,(\mu -1)\}_p
\ul R^{-\ul\a} \; .
\eeq
\end{cor}

\begin{cor}\label{cor:DworkRobba2} Suppose the matrices $G_i$ are holomorphic and bounded in $D(a,r^-)$.  Let $\xi \in D(a,r^-)$, let $\xi^\p \in D^d_{\sH(\xi)}(0,r^-)$ be the canonical point above $\xi$ and let us assume that the fundamental solution matrix (\ref{eq:taylorseriesintro}) of (\ref{eq:intrdiffsys}) at $\xi^\p$ converges in the polydisk $D_{\sH(\xi)}(\xi^\p,(R_1,\dots,R_d)^-)\subseteq D_{\sH(\xi)}(a,r^-)$. Let
$$
C=\max_{|\ul\be|_\infty\lneqq\mu}
\l(\ul R^{\ul\be}
\l|\l|{G}_{\ul\be}\r|\r|_{D(a,r^-)} \r)\  .
$$
Then, for any $\ul\a\in\N^d$ we have the
following estimate
\beq\label{eq:DworkRobba2}
|{G}_{[\ul\a]}(\xi)| \leq C\{|\ul\a|_\infty,(\mu -1)\}_p
\ul R^{-\ul\a} \; .
\eeq
\end{cor}

The proof of (\ref{thm:DworkRobbaBerk}) now follows directly. We consider $\xi \in U$, and the canonical point $\xi^\p \in U_{\sH(\xi)}$ above it; let $R = R(\xi,\Sigma) \leq 1$. The disk $D_{\cX}(\xi,R^-) \subset U_{\sH(\xi)}$ is isomorphic via the coordinate functions to $D_{\sH(\xi)}(\ul x(\xi), (R,\dots,R)^-)$. We
apply  (\ref{cor:DworkRobba2}) to the restriction of $\Sigma$ to $D_{\cX}(\xi,R^-)$, taking $r = R$.

\subsection{Upper semicontinuity of $\xi \mapsto  R(\xi,\Sigma)$}
\label{subsection:USC}

We are now back to  the assumptions in  (\ref{situation}) and (\ref{thm:DworkRobbaBerk}), so in particular the matrices $G_i$ are supposed to be {\it bounded} on $U$, and let us further insist that the function $\xi \mapsto \delta_{\cX}(\xi,U)$ be USC on $U$.
For example, by (\ref{prop:diamlaurent}), this happens when
$U$ be a Laurent domain in $X$.
For $s = 1,2, \dots$ and for $\xi \in U$, let
\beq \varphi_s(\xi) =   \min (\delta_{\cX}(\xi,U), \inf_{|\ul \a|_{\infty} \geq s} |G_{[\ul \a]}(\xi)|^{-1/|\ul \a|_{\infty}}) = \inf_{|\ul \a|_{\infty} \geq s} \min ( \delta_{\cX}(\xi,U), |G_{[\ul \a]}(\xi)|^{-1/|\ul \a|_{\infty}})
\; .
\eeq
So, $\eta \mapsto \varphi_s(\xi)$ is USC on $U$, and
\beq  R_{\cX}(\xi, \Sigma) ) = \lim_{s \to \infty} \varphi_s(\xi)   \; ,
\eeq
where $R_{\cX}(\xi, \Sigma) )$ is  the function introduced in (\ref{eq:radiusdef}).
 The  corollary (\ref{uniform}) of the Dwork-Robba theorem  says that, $\forall\, \veps >0$,  $\exists \,s_{\veps}$ such that $\forall \,\ul \a$ with $|\ul \a|_{\infty} \geq s_{\veps}$
\beq
  |G_{[\ul \a]}(\xi)|^{1/|\ul \a|_{\infty}} \leq (1+\veps)/R_{\cX}(\xi, \Sigma) \; ,\;\; \forall \, \xi \in U \; .
\eeq
So,
\beq
 |G_{[\ul \a]}(\xi)|^{ - 1/|\ul \a|_{\infty}} \geq \frac{ R_{\cX}(\xi, \Sigma) } {1+\veps} \; ,\;\; \forall \,  \xi \in U \; .
\eeq
Hence
\beq \begin{array} {cccc}
\forall \, \veps >0 & \exists\, s_{\veps} & \hbox{such that} & \forall \, s \geq s_{\veps}
\\
\\
 \varphi_s(\xi) \leq &R_{\cX}(\xi, \Sigma) & \leq (1 + \veps)  \varphi_s(\xi)  &\forall \; \xi \in U\;  ,
\end{array}
\eeq
because the sequence $s \mapsto \varphi_s$ is an increasing sequence of functions on $U$. Then, $\forall\, \veps >0$,  $\exists \,s_{\veps}$ such that
\beq 0 \leq R_{\cX}(\xi, \Sigma)  - \varphi_s(\xi) \leq \veps \;\; \forall \xi \in  U\; ,\;\;  \forall\, s \geq s_{\veps}\;.
\eeq
Then $\xi \mapsto R_{\cX}(\xi, \Sigma) $ is a uniform limit of USC functions, and is therefore USC.
We then state
\begin{thm} Assume the matrices $G_i$ are bounded analytic functions  on the analytic domain $U$, and suppose the function $\xi \mapsto \delta_{\cX}(\xi,U)$ is USC on $U$. Then $\xi \mapsto R_{\cX}(\xi, \Sigma) $ is   USC on $U$.
\end{thm}
Similarly we have
\begin{thm} Assume the matrices $G_i$ are  analytic elements  on the analytic domain $U$, and suppose the function $\xi \mapsto \delta_{\cX}(\xi,U)$ admits a continuous extension to ${\rm cl}_X(U)$. Let us define $R_{\cX}(\xi, \Sigma) $ on ${\rm cl}_X(U)$ as   in theorem (\ref{thm:DRanelem}).  Then $\xi \mapsto R_{\cX}(\xi, \Sigma) $ is  USC on ${\rm cl}_X(U)$.
\end{thm}
\subsection{Continuity of $\xi \mapsto  R(\xi,\Sigma)$ at maximal points (Dwork's transfer theorem)}
\label{subsection:CONTmax}
An immediate consequence of formula (\ref{eq:pracDR}) is the following. Let $U$ be an affinoid domain in $X$, and let $\Gamma(U) = \{\eta_1,\dots,\eta_N\}$ be the Shilov boundary of $U$. Then, for the constant $C = C(\Sigma,U)$ of
(\ref{thm:DworkRobbaBerk}),
\beq \label{eq:TRANS}
||G_{[\ul\a]}||_U \leq C |\ul\a|_\infty^{\mu -1}
\l(\min_{i=1,\dots,N}R(\eta_i,\Sigma)\r)^{- |\ul\a|_\infty} \; .
\eeq
This shows that, for any $\xi \in U$,
\beq \label{eq:EST}
\wtilde R(\xi,\Sigma) \geq \min_{i=1,\dots,N}R(\eta_i,\Sigma) \; .
\eeq
\begin{prop} Let us assume that $U$ is a Laurent domain in $X$ with a unique maximal point $\eta_U$, and  that the function $\xi \mapsto  \delta_{\cX}(\xi,U)$ be continuous at $\eta_U$. Then,
$\xi \mapsto R(\xi,\Sigma)$ is continuous at $\eta_U$.
\end{prop}
\begin{proof}

USC of $\xi \mapsto R(\xi,\Sigma)$, shows that
$$
\lim_{\xi \to \eta_U}  R(\xi,\Sigma) \leq  R(\eta_U,\Sigma)\; .
$$
Since
$$
\lim_{\xi \to \eta_U}  \delta_{\cX}(\xi,U) =  \delta_{\cX}(\eta_U,U) \;  ,
$$
we conclude by (\ref{eq:EST}) that
$$
\lim_{\xi \to \eta_U} R(\xi,\Sigma) = R(\eta_U,\Sigma)\; .
$$
\end{proof}
\begin{cor} \label{eq:contGEN} If $U = X$, $\xi \mapsto R(\xi,\Sigma)$ is continuous at $\eta_X$.
\end{cor}
\begin{proof}
$\delta_{\cX}(\xi,X) = 1$, $\forall \, \xi \in X$.
\end{proof}
\section{The one-dimensional case}

\subsection{The theorem of Christol-Dwork revisited}

Christol and Dwork (\cf \cite{ChristolDwork}) consider a differential system
\beq \label{diffCHDW}
\Sigma = \Sigma_{x,G,U} \;\; : \;\; \frac d{dx} \vec y=G\,\vec y
\eeq
with $G$ a $\mu\times\mu$ matrix of  analytic elements
on the annulus
$$
C(r_1,r_2)=\l\{\xi:r_1<|x(\xi)|<r_2\r\}
\subset D_k(0,1^+)\,.
$$
So, the entries of $G$ are elements of the $k$-Banach algebra $\sH(r_1,r_2)$ of uniform limits on $C(r_1,r_2)$ of rational functions in $k(x)$, having no pole in  $C(r_1,r_2)$.
Here
$$
{\rm cl}_X(U) = C^+(r_1,r_2) = C(r_1,r_2) \cup \{t_{0,r_1}, t_{0,r_2} \}\,.
$$
Christol and Dwork  consider the function radius of convergence of (\ref{diffCHDW}), restricted to a segment of points in $C^+(r_1,r_2)$, namely
$$
\begin{array}{cccc}  R:&
[r_1,r_2]&\longrightarrow & \R_{\geq 0}\\
&r & \longmapsto &  R(r):= R(t_r,\Sigma) = \min (r,\wtilde R(t_r,\Sigma))
\end{array}\,,
$$
where $t_r=t_{0,r}$ is the point at the boundary of $D(0,r^-)$. This coincides with our definitions (\ref{thm:DRanelem}) taking into account the fact that the function $\xi \mapsto |x(\xi)|$ extends continuously $\xi \mapsto \delta_{\cX}(\xi,C(r_1,r_2))$ to $C^+(r_1,r_2)$.
\par
So the problem is to describe
$$
r\longmapsto \wtilde R(r,\Sigma)=\liminf_{s\to\infty}\l|G_{[s]}(t_r)\r|^{-1/s}
$$
on $[r_1,r_2]$. They use the well-known fact that, for any $f \in \sH(r_1,r_2)$, the function
$\rho \mapsto \log |f (t_{e^{\rho}})|$ is convex and continuous on the interval $[r_1,r_2]$.
It is an elementary fact that,
if $\forall i \in \N$, $\varphi_{i}: [r_1,r_2] \to \R$
is a convex (resp. concave) function, then
$$
\varphi = \limsup_{i \to \infty} \varphi_i \;\; {\rm ( resp.}\;\; \varphi = \liminf_{i \to \infty} \varphi_{i}  \; {\text )}\;
$$
is convex (resp. concave).
They conclude that
the function
$\rho\mapsto \log \wtilde R(e^\rho)$ is concave (\ie $\cap$-shaped) in
$[\log r_1,\log r_2]$.
So the function $\wtilde R$ is continuous in $(\log r_1,\log r_2)$ and LSC
at $\log r_1$ and $\log r_2$. Then the same is true for the function $R$.
But we have proven in section (\ref{subsection:USC}),
that the function $R$ is USC in $U$, so,
in the present case,  $R$ is continuous.
The conclusion is that:

\begin{thm}[Christol-Dwork] \label{thm:ChDw} Let $\cX = \what \A^1_{\kc}$,  $U = C(r_1,r_2)$, and assume the entries of $G$ are analytic elements on $C(r_1,r_2)$.
Then the  function
$$
\begin{array}{ccc}
[r_1,r_2]&\longrightarrow & \R_{>0}\\ \\
r & \longmapsto &  R(r) = R(t_{0,r},\Sigma)
\end{array}
$$
is continuous.
\end{thm}
\begin{rmk}
We do not claim that the  function $r \mapsto \wtilde R (r,\Sigma)$ is continuous at $r_1$, $r_2$.
\end{rmk}

\subsection{Continuity of $\xi \mapsto R_{\cX}(\xi,\Sigma)$ on an affinoid $U \subset D(0,1^+)$ of dimension 1}
\label{proofCONT}
In this section we consider a system $\Sigma = \Sigma_{x,G,U}$ of the form (\ref{diffCHDW}) on an affinoid  domain $U$ of $D(0,1^+)$.
So, this is the case of system (\ref{eq:intrdiffsys}) under the assumptions of (\ref{situation}),  in dimension one, and with the further condition that  $U$ is affinoid and $\cX=\what{\A}^1_{k^\circ}$.
\par

We prove continuity of $\xi \mapsto R(\xi,\Sigma)$. Since our definitions of diameter and radius of convergence of a system are invariant by base-field extension, we may apply the discussion of  (\ref{subsec:ground}), and assume that the field $k$ is maximally complete and
algebrically closed. An affinoid $U$ of $D(0,1^+)$
is of the form
\beq \label{specialset}
U=D(0,1^+)\smallsetminus\cup_{i\in I}D(a_i,r_i^-)\subset X\,,
\eeq
where $I$ is a finite set and $a_i$ is a  $k$-rational point of $D(0,1^+)$. We are left to prove LSC continuity of $\xi \mapsto R(\xi,\Sigma)$ for this  system.
\par
Notice that, because $k$ is maximally complete and
algebrically closed, the points of $D(0,1^+)$ are either $k$-rational points or of the form  $t_{a,r}$ = the boundary point of a disk $D(a,r^-)$, centered at
a $k$-rational  point $a$ and of radius $r \in (0,1]$.

\begin{thm} \label{thm:onedimCONT}
The function $\xi \mapsto R(\xi,\Sigma)$ is  continuous on $U$.
\end{thm}
\begin{proof}
\par

We will have to restrict the system $\Sigma$ to various affinoid subdomains $V$ of $U$. We then write  $R(\xi,V)$ for  $R(\xi,\Sigma)$, when $\Sigma$ is restricted to the affinoid $V \subset U$.
Let $\xi \in U$ be a $k$-rational point. Then the function  $\eta \mapsto R(\eta, U)$, which expresses the radius of the maximal open disk centered at $\eta$ and contained in $U$ on which $\Sigma$ is the trivial differential system, is clearly constant in the neighborhood $D_k(\xi,R(\xi,U)^-)$ of $\xi$. This neighborhood is non empty since  $R(\xi,U) = \min(\wtilde R(\xi,\Sigma), \delta_{\cX}(\xi,U))$, and we compute:
 \beq
 \delta_{\cX}(\xi,U) = \min_{i \in I} |x(\xi) - a_i| >0 \; ,
 \eeq
  and,  by (\ref{prop:trivialestimate}),
\beq \wtilde R(\xi,\Sigma)   \geq \min (1,
\frac{|p|^{\frac{1}{p-1}}}
{\max \l(|\frac{d}{dx}|_U,||G||_U\r)} ) >0\, .
\eeq
We are left to prove  continuity (in fact just  LSC) of  $\xi \mapsto R(\xi,U)$ at a  point $\xi \in U$ of the form
 $\xi=t_{a,r}\in U$. Notice that we may (and will) assume $R = R(t_{a,r},U) \leq r$, otherwise on the disk $D(a,R^-)$, which is an open neighborhood of  $t_{a,r}$, the function $\xi \mapsto R(\xi,U)$ would be constant of value $R$, and $\xi \mapsto R(\xi,U)$ would then be continuous at $\xi=t_{a,r}$. On the other hand, $\delta_{\cX}(t_{a,r},U) \geq r$, so in particular we assume $ \wtilde R(t_{a,r},\Sigma) \leq \delta_{\cX}(t_{a,r},U)$.
 \par
 Let
$$
J = \l\{i\in I:|a-a_i|=r \r\}\,,
$$
and let $\veps_0 := \min_{i \notin J}  |a-a_i|$. We further subdivide $J$ into a disjoint union $J = J_1 \cup J_2$, where
$$
J_1 = \l\{i\in J: r_i < r \r\} \;\;\; , \;\;\; J_2 = \l\{i\in J: r_i = r \r\} \; .
$$
 We want to construct a (not fundamental) system of affinoid neighborhoods $\{ V_\veps \}_{\veps > \veps_0}$ of $t_{a,r} \in U$, with the property that the Shilov boundary $\Gamma(V_\veps)$ of $V_\veps$
consists of the points $t_{a_i,r - \veps}$, for $i \in J_1$, $t_{a_i,r } = t_{a,r } $, for $i \in J_2$ and of the point $t_{a,r + \veps}$. Notice that $t_{a,r + \veps} \to t_{a,r }$, and $t_{a_i,r - \veps} \to t_{a,r }$,  as $\veps \to 0$, $\forall \, i \in J_1$. We simply take for $\veps >  \veps_0$
\beq
V_\veps=\l\{\eta\in X:r-\veps\leq|x(\eta)-a|\leq r+\veps\r\} \setminus \l(\bigcup_{i \in J_1}D(a_i, (r - \veps)^-)
\cup
\bigcup_{i \in J_2}D(a_i, r^-) \r)
 \; .
\eeq
Notice that
\beq
 \delta_{\cX}(t_{a,r+\veps}, V_\veps) = r + \veps \; \;  ; \; \;
\delta_{\cX}(t_{a,r}, V_\veps) = r \; \;  ; \; \;
\delta_{\cX}(t_{a_i,r-\veps}, V_\veps) = r-\veps \;\; \forall i \in J\; .
\eeq

Coming back to our differential system (\ref{diffCHDW}) and its iterates
$$
\frac 1{s!}\l(\frac d{dx}\r)^s Y=G_{[s]}Y\,,
\hbox{ $G_{[s]}\in M_{\mu\times\mu}(\cO(U))$,}
$$
we have, by (\ref{eq:EST}), $\forall \eta \in V_\veps$,
$t_{a_i,r - \veps}$, for $i \in J_1$, $t_{a_i,r } = t_{a,r } $, for $i \in J_2$ and of the point $t_{a,r + \veps}$.
$$
\wtilde R (\eta,\Sigma) \geq \, \min \, (\, \min_{i \in J} R  (t_{a_i,r-\veps},V_\veps) \, ,  \, R  (t_{a,r},V_\veps))
\,.
$$
The  affinoid  $V_\veps$   contains  the annuli
$\{ \eta \in X \, | \, r< |x(\eta)-a| < r +\veps \}$, $\{ \eta \in X \, | \, r - \veps < |x(\eta)-a_i| < r  \}$
 and analytic functions on $V_\veps$ restrict to analytic elements on them. So, we may apply the
 theorem of Christol-Dwork (\ref{thm:ChDw}) to deduce
$$
\lim_{\veps\to 0}
 R  (t_{a_i,r-\veps}, V_\veps) =   R (t_{a,r},V_\veps)\, ,
$$
$\forall \, i \in J$, and similarly $\displaystyle
\lim_{\veps\to 0}
 R  (t_{a,r+\veps}, V_\veps) =   R (t_{a,r},V_\veps)$.
 Notice that
 $$
\lim_{\veps\to 0}
 \delta_{\cX} (t_{a_i,r-\veps}, V_\veps) =  \lim_{\veps\to 0}
 \delta_{\cX} (t_{a,r+\veps}, V_\veps) =  \delta_{\cX}  (t_{a,r}, V_\veps) = r\, ,
$$
$\forall \, i \in J$.
We conclude that $\forall \, \sigma >0$, $\exists \, \veps > \veps_0$ such that
$$
\wtilde R (\eta,\Sigma) \geq  R  (t_{a,r},V_\veps) - \sigma = \min (\wtilde R (t_{a,r},\Sigma), r) - \sigma =
\wtilde R (t_{a,r},\Sigma) - \sigma
\, ,
$$
$\forall \, \eta \in V_\veps$.

\par
The conclusion is that  $\eta \mapsto \wtilde R  (\eta,\Sigma )$ is LSC at $t_{a,r}$. Since in the present case, $\eta \mapsto \delta_{\cX}(\eta,U)$ is continuous, we conclude that $\eta \mapsto R  (\eta,U)$ is LSC at $t_{a,r}$.
Since we already know that it is USC, we conclude that it is actually
continuous at $t_{a,r}$.
\end{proof}

\subsection{Continuity of $\xi \mapsto R_{\cX}(\xi,\Sigma)$ in $\dim X = 1$, when  $U $ is a neighborhood of $\eta_X$}
\label{onedimCONT}
We assume here that $\cX$ is a smooth formal scheme of relative dimension 1 over $\Spf \, \kc$, and $U$ is an affinoid neighborhood of the maximal point $\eta_X$ of $X = \cX_{\eta}$ (always satisfying the requirements in (\ref{situation})). Notice that the special fiber $\cX_s$ of $\cX$ is a smooth scheme over $\kt$, which we may assume to be connected. The reduction map $\pi : X \to \cX_s$, is such that the fiber at each closed point of $\cX_s$ is an open disk of radius one, called {\it a residue class}, while the inverse image of the generic point $\eta_{\cX_s}$ consists only of the maximal point $\eta_X$ of $X$. An affinoid $U \subset X$ is a neighborhood of $\eta_X$ if and only if it contains almost all residue classes in $X$, and contains a non trivial annulus of outer radius one in each of the remaining residue classes.
So, $U$ is the disjoint union of the generic fiber $Y = \cY_{\eta}$ of a smooth formal scheme $\cY$ (the union of full residue classes) and of a finite number $\{ C_1, \dots,C_r \}$ of analytic subdomains of the open disk of radius one, which are bound to contain some annulus of outer radius one.
We are given the system (\ref{diffCHDW}) on $U$, and must prove continuity of $\xi \mapsto R_{\cX}(\xi,\Sigma)$ on $U$. Notice that, if we call $\Sigma_{|Y}$, $\Sigma_{|C_1}$, \dots, $\Sigma_{|C_r}$, the restrictions of $\Sigma$ to the various analytic subdomains of $U$, we have
$$R_{\cY}(\xi,\Sigma_{|Y}) = R_{\cX}(\xi,\Sigma) \; , \; \forall \xi \in Y \; ,
$$
because $\delta_{\cX}(\xi,Y) = \delta_{\cY}(\xi,Y) = 1$, $\forall \, \xi \in Y$.
Similarly,
$$R_{\cY}(\xi,\Sigma_{|C_i}) = R_{\cX}(\xi,\Sigma) \; , \; \forall \xi \in C_i \; , \; \forall \, i =1,\dots,r \;.
$$
We already proved continuity of $\xi \mapsto R_{\cY}(\xi,\Sigma_{|Y})$ as a function $Y \to \R$, at  the maximal point $\eta_Y = \eta_X$. Notice that $R_{\cY}(\xi,\Sigma_{|Y}) = \min (1, \wtilde R(\xi,\Sigma))$, since $\delta_{\cY}(\xi,Y)=1$. Continuity on $Y$ then follows from theorem (\ref{thm:onedimCONT}), since the residue classes in $Y$ may be considered independently, except for being glued at $\eta_Y$, and the definition of $R_{\cY}(\eta_Y,\Sigma_{|Y}) = \min (1, \wtilde R(\eta_Y,\Sigma))$ only depends upon $\Sigma$ viewed as a differential system on $\sH(\eta_Y)^{\mu} = \sH(\eta_X)^{\mu}$. As for the residue classes containing $C_1$, \dots, $C_r$, respectively, let us fix one of them, which we regard as $D(0,1^-) \supset C_1$. Notice that $C_1$ contains some open annulus of outer radius 1. Continuity of  $\xi \mapsto R_{\cX}(\xi,\Sigma_{|C_1})$ on $C_1$ also follows from (\ref{thm:onedimCONT}).
We have to prove that
\beq
\lim_{\xi \to \eta_X}  R_{\cX}(\xi,\Sigma_{|C_1}) = R_{\cX}(\eta_X,\Sigma) \; .
\eeq
Since $\delta_{\cX}(\xi,U)$ is continuous on $U$ and takes the value 1 at $\eta_X$,
restricting  to points $\xi$ of the form $t_{0,r}$, as $r \to 1$,  it follows  from (\ref{thm:ChDw}) that
\beq
\lim_{\xi \to \eta_X}    R_{\cX}(\xi,\Sigma_{|C_1}) = \min(1,\wtilde R(\eta_X,\Sigma)) \; ,
\eeq
which only depends upon $\eta_X$, and is therefore independent of the residue class containing  $C_1$, chosen to approach $\eta_X$.
We conclude
\begin{cor} \label{eq:contLISCIA} If $U$ is an affinoid neighborhood of $\eta_X$, $\xi \mapsto R(\xi,\Sigma)$ is continuous on $U$.
\end{cor}

\section*{Appendix.
Valuation polygon of an analytic function in several
variables}
\addcontentsline{toc}{section}{Appendix. Valuation polygon of an analytic function in several
variables}
\label{sec:poligono di valutazione}

For the reader's convenience we recall some facts from
\cite[\S2]{RobbaBSMF} on the several variable Newton polygon theory,
which has been applied in this paper.
\par

We set
$v(x)=-\log |x|$ for any $x$ in an extension of $k$. Let us assume that $k$ is algebraically closed. For any convergent power
series $f=\sum_{\ul\a\in\N^d}a_{\ul\a}\ul x^{\ul\a}\in  k \[[\ul
x\]]= k\[[x_1,\dots,x_d\]]$, \ie for a formal power series such that
$$
\liminf_{n\to\infty}\frac{v(a_{\ul\a})}{|\ul\a|_\infty}>-\infty\,,
$$
we set:
\begin{eqnarray}
&{\rm Conv}(f)=\l\{\ul\mu\in\R^d:v(a_{\ul\a})+\sum{\a_i\mu_i}\to+\infty
\hbox{ when } \sum\a_i\to +\infty\r\}\,;\\ \nonumber\\
&\ds v(f,\ul\mu)=\inf_{\ul\a\in\N^d,f_{\ul\a}\neq
0}\l(v(f_{\ul\a})+\sum_i\a_i\mu_i\r)\,,
\hbox{ for any $\ul\mu\in {\rm Conv}(f)$;}\\\nonumber\\
&{\rm Reg}(f)=\l\{\ul\mu\in {\rm Conv}(f):\exists!\ul\be\in\N^d \hbox{ s.t. }
v(f,\ul\mu)=v(a_{\ul\be})+\sum_i\be_i\mu_i\r\}\,;\\\nonumber\\
&{\rm Z}(f)={\rm Conv}(f)\smallsetminus {\rm Reg}(f)\,.
\end{eqnarray}
Then the following properties hold:
\begin{enumerate}

\item
${\rm Conv}(f)$ is a convex subset of $\R^d$;

\item
$v(f,-)$ is a concave continuous function on ${\rm Conv}(f)$;

\item
the graph of $v(f,-)$ on the interior of ${\rm Conv}(f)$ is a polyhedron
(with possibly infinitely many faces).

\end{enumerate}

\begin{prop}[{\cf \cite[2.12,2.20]{RobbaBSMF}}]~~
\begin{enumerate}
\item
Let $\ul \xi\in k^d$ and $\ul\mu=(v(\xi_1),\dots,v(\xi_d))$ be in
${\rm Conv}(f)$. If $\ul\mu\in {\rm Reg}(f)$ then $f(\ul\xi)\neq 0$ and
$v(f(\ul\xi))=v(f,\ul\mu)$.

\item
Let $\ul\mu\in {\rm Z}(f)\cap v(k)^d$. Then there exists $\ul\xi\in
k^d$ such that $f(\ul\xi)=0$ and $v(\xi_i)=\mu_i$ for any
$i=1,\dots,d$.

\end{enumerate}
\end{prop}

\begin{cor}
\label{cor:Robba} Let $\ul \xi\in k^d$ with
$\ul\mu=(v(\xi_1),\dots,v(\xi_d))\in Conv(f)$. We suppose that
$|f(\ul\xi)|< |f(\ul 0)|$ (resp. $|f(\ul\xi)|>|f(\ul 0)|$). Then there
exists $\ul\zeta\in  k^d$ such that $f(\ul\zeta)=0$ and
$|\zeta_i|=|\xi_i|$ (resp. $|\zeta_i|\leq|\xi_i|$) for any
$i=1,\dots,d$.
\end{cor}

\begin{proof}
Let us suppose that $|f(\ul\xi)|<|f(0)|$. This means that
$v(f_{(0,\dots,0)})<v(f(\xi))$ and hence that
$$
\inf_{\ul\a\in\N^d,f_{\ul\a}\neq
0}\l(v(f_{\ul\a})+\sum_i\a_i\mu_i\r)=
v(f_{\ul\a^\p})+\sum_i\a_i^\p\mu_i=v(f_{\ul\a^{\p\p}})+\sum_i\a_i^{\p\p}\mu_i\,,
$$
for some $\ul\a^\p,\ul\a^{\p\p}\in\N^d$ such that
$\ul\a^\p\neq\ul\a^{\p\p}$. Therefore $\ul\mu\in {\rm Z}(f)$ and the
corollary follows from the previous proposition.
\par
If on the contrary $|f(\ul\xi)|>|f(\ul 0)|$ it is enough to consider the
expansion of $f$ at $\ul\xi$.
\end{proof}


\end{document}